\documentclass[10.9pt,a4paper]{amsart}
\usepackage[english]{babel}
\newtheorem{theorem}{Theorem}[section]

\newtheorem{corollary}[theorem]{Corollary}

\newtheorem{proposition}[theorem]{Proposition}

\theoremstyle{definition}
\newtheorem*{definition}{Definition}

\newtheorem{remark}[theorem]{Remark}

\usepackage[square,sort,comma,numbers]{natbib}

\usepackage[colorlinks,
    linkcolor={red!50!black},
    citecolor={blue!50!black},
    urlcolor={blue!80!black}]{hyperref}

\usepackage[all]{xy}
\usepackage{pstricks}
\usepackage{enumerate}
\usepackage{amsfonts,amssymb,amsmath,eucal,pinlabel,array,hhline}
\usepackage{slashed}
\usepackage{tabulary}
\usepackage{fancyhdr}
\usepackage{a4wide}
\usepackage{calrsfs,bbm,dsfont}
\usepackage[position=b]{subcaption}

\newcommand{\Z}{\mathds{Z}}
\newcommand{\Q}{\mathds{Q}}
\newcommand{\R}{\mathds{R}}
\newcommand{\C}{\mathds{C}}
\newcommand{\LR}{\R[t^{\pm 1}]}

\newcommand{\id}{\operatorname{id}}
\newcommand{\tr}{\operatorname{tr}}
\newcommand{\Bl}{\operatorname{Bl}}
\newcommand{\bl}{\operatorname{bl}}
\newcommand{\SU}{\operatorname{SU}}
\newcommand{\sign}{\operatorname{sign}}

\DeclareSymbolFont{EulerScript}{U}{eus}{m}{n}
\DeclareSymbolFontAlphabet\mathscr{EulerScript}
\begin{document}

\title{The Levine-Tristram signature: a survey}
\author{Anthony Conway}
\address{Department of Mathematics, Durham University, United Kingdom}
\email{anthonyyconway@gmail.com}
%\date{\today}
\maketitle
\begin{abstract}
%Given an oriented link $L$, the Levine-Tristram signature is a function $\sigma_L \colon S^1 \to \Z$. 
%Some of its numerous applications include the study of unlinking numbers and link concordance.
%This survey aims to collect the three and four dimensional definitions of $\sigma_L$ and lists some its properties and applications.
%In particular, this article does not contain any new result.
The Levine-Tristram signature associates to each oriented link $L$ in $S^3$ a function $\sigma_L \colon S^1 \to \Z.$ 
This invariant can be defined in a variety of ways, and its numerous applications include the study of unlinking numbers and link concordance.
In this survey, we recall the three and four dimensional definitions of $\sigma_L$, list its main
properties and applications, and give comprehensive references for the proofs of these statements.
\end{abstract}

\section{Introduction}

Given an oriented link $L \subset S^3$, the Levine-Tristram signature is a function $\sigma_L\colon S^1 \to \Z$ whose study goes back to the sixties~\cite{Tristram, LevineKnotCob}.
The main goal of this survey article is to collect the various definitions of $\sigma_L$, while a secondary aim is to list its properties.
Although some elementary arguments are outlined in the text, we provide detailed external references for most of the proofs.
Briefly, we will discuss the definition in terms of Seifert matrices, various 4-dimensional interpretations as well as connections to pairings on the Alexander module.
The next paragraphs give the flavor of some of these constructions.

Most knot theory textbooks that cover the Levine-Tristram signature introduce it using Seifert matrices~\cite{Lickorish, KawauchiBook, KauffmanOnKnots, LivingstonBook}.
Indeed, as we review in Section~\ref{sec:DefLevineTristram}, the Levine-Tristram signature at~$\omega \in S^1$ can be defined using any Seifert matrix $A$ for $L$ by setting
$$ \sigma_L(\omega)=\operatorname{sign}  (1-\omega)A+(1-\overline{\omega})A^T.$$
In the same section, we collect the numerous properties of $\sigma_L$: after listing its behavior under mirror images, orientation reversals and satellite operations, we review applications to unlinking numbers, link concordance and discuss various incarnations of the Murasugi-Tristram inequality~\cite{Murasugi, Tristram}.

The signature admits several 4-dimensional interpretations: either using covers of $D^4$ branched along surfaces cobounding $L$~\cite{Viro}, or applying twisted signatures, or as invariants of the zero framed surgery along $L$. 
%These constructions are discussed in Section~\ref{sec:4DDefinitions}, but we sketch one of them briefly.
Before discussing these constructions in detail in Section~\ref{sec:4DDefinitions}, let us briefly sketch one of them.
Given a locally flat compact connected oriented surface $F \subset D^4$ with boundary~$L$, we set $W_F:=D^4 \setminus \nu F$ and consider the coefficient system $ \pi_1(W_F) \to H_1(W_F) \cong \Z \to \C$ which maps the meridian of $F$ to $\omega$.
This gives rise to a twisted intersection form $\lambda_{\C^\omega}(W_F)$ on the twisted homology $\C$-vector space~$H_2(W_F;\C^\omega)$ whose signature coincides with  the Levine-Tristram signature:
$$ \sigma_L(\omega)=\operatorname{sign} \lambda_{\C^\omega}(W_F).$$
Section~\ref{sec:3DDefOther} is concerned with methods of extracting $\sigma_K(\omega)$ from pairings on the Alexander module $H_1(X_K;\Z[t^{\pm 1}])$ of a knot $K$ (here we write $X_K:=S^3 \setminus \nu K$ for the exterior of $K$)~\cite{MilnorInfiniteCyclic, KeartonSignature}.
Briefly, the signature~$\sigma_K$ can be extracted by considering the primary decomposition of $H_1(X_K;\R[t^{\pm 1}])$ and by studying the Milnor pairing or the Blanchfield pairing
$$ H_1(X_K;\Z[t^{\pm 1}]) \times H_1(X_K;\Z[t^{\pm 1}]) \to \Q(t) /\Z[t^{\pm 1}].$$

In fact, as we discuss in Section~\ref{sec:Other}, the signature can also be understood as a signed count of~$\SU(2)$ representations of $\pi_1(X_K)$ with fixed meridional traces~\cite{Lin, HeusenerKroll}, or in terms of the Meyer cocycle and the Burau representation~\cite{GambaudoGhys}.
Summarizing, $\sigma_L$ admits a wealth of definitions, which never seemed to have been collected in a single article.

We conclude this introduction with two remarks. Firstly, note that we neither mention the Gordon-Litherland pairing~\cite{GordonLitherland} nor the multivariable signature~\cite{CimasoniFlorens}. 
Secondly, we stress that even though $\sigma_L$ was defined 50 years ago, it continues to be actively studied nowadays. 
We mention some recent examples: 
results involving concordance properties of positive knots can be found in~\cite{BaaderDehornoyLiechti};
 the behavior of~$\sigma_L$ under splicing is now understood~\cite{DegtyarevFlorensLecuona};
  the relation between the jumps of~$\sigma_L$ and the zeroes of~$\Delta_L$ has been clarified~\cite{GilmerLivingston, Liechti};
  a diagrammatic interpretation of $\sigma_L$ (inspired by quantum topology) is conjectured in~\cite{Kashaev};
there is a characterization of the functions that arise as knot signatures~\cite{LivingstonCharac};
new lower bounds on unknotting numbers have been obtained via~$\sigma_K$~\cite{LivingstonUnknotting};
there is a complete description of the~$\omega \in S^1$ at which~$\sigma_L$ is a concordance invariant~\cite{NagelPowell};
and $\sigma_L$ is invariant under \emph{topological} concordance~\cite{PowellSignature}.

\medbreak
This survey is organized is as follows. In Section~\ref{sec:DefLevineTristram}, we review the Seifert matrix definition of $\sigma_L$ and list its properties. In Section~\ref{sec:4DDefinitions}, we outline and compare the various four dimensional interpretations of $\sigma_L$. In Section~\ref{sec:3DDefOther}, we give an overview of the definitions using the Milnor and Blanchfield pairings.
In Section~\ref{sec:Other}, we outline additional constructions in terms of $\SU(2)$ representations and braids.
%\tableofcontents

\section{Definition and properties}
\label{sec:DefLevineTristram}

In this section, we review the definition of the Levine-Tristram and nullity using Seifert matrices (Subsection~\ref{sub:DefSeifert}) before listing several properties of these invariants (Subsections~\ref{sub:ElemProperties},~\ref{sub:Unlinking} and~\ref{sub:Concordance}). Knot theory textbooks which mention the Levine-Tristram signature include~\cite{Lickorish, KawauchiBook, KauffmanOnKnots, LivingstonBook}.

\subsection{The definition via Seifert surfaces}
\label{sub:DefSeifert}

A \textit{Seifert surface} for an oriented link $L$ is a compact oriented surface $F$ whose oriented boundary is $L$. While a Seifert surface may be disconnected, we require that it has no closed components. Since $F$ is orientable, it admits a regular neighborhood homeomorphic to $F \times [-1,1]$ in which $F$ is identified with $F \times \lbrace 0 \rbrace$. For $\varepsilon=\pm 1$, the \emph{push off maps} $i^\varepsilon \colon H_1(F;\Z) \rightarrow H_1(S^3\setminus F;\Z)$ are defined by sending a (homology class of a) curve $x$ to~$i^\varepsilon(x):=x \times \lbrace \varepsilon \rbrace$. The \textit{Seifert pairing} of $F$ is the bilinear form
\begin{align*}
 H_1(F;\Z) \times H_1(F;\Z) & \rightarrow \Z \\
 (a,b) &\mapsto \ell k(i^-(a),b).
\end{align*}
A \textit{Seifert matrix} for an oriented link $L$ is a matrix for a Seifert pairing. Although Seifert matrices do not provide link invariants, their so-called \textit{S-equivalence class} does~\cite[Chapter 8]{Lickorish}. 
Given a Seifert matrix $A$, observe that the matrix $(1-\omega)A+(1-\overline{\omega})A^T$ is Hermitian for all $\omega$ lying in $S^1$.

\begin{definition}
\label{def:LevineTristram}
Let $L$ be an oriented link, let $F$ be a Seifert surface  for $L$ with $\beta_0(F)$ components and let $A$ be a matrix representing the Seifert pairing of $F$. Given $\omega \in S^1$, the \textit{Levine-Tristram signature} and \emph{nullity} of $L$ at $\omega$ are defined as
\begin{align*}
 \sigma_L(\omega)&:=\sign((1-\omega)A+(1-\overline{\omega})A^T), \\
 \eta_L(\omega)&:=\operatorname{null}((1-\omega)A+(1-\overline{\omega})A^T)+\beta_0(F)-1. 
 \end{align*}
\end{definition}

These signatures and nullities are well defined (i.e. they are independent of the choice of the Seifert surface)~\cite[Theorem 8.9]{Lickorish} and, varying $\omega$ along $S^1$, produce functions $ \sigma_L,\eta_L \colon S^1 \rightarrow \Z$.
The Levine-Tristram signature is sometimes called the \emph{$\omega$-signature} (or the \emph{equivariant signature} or the \emph{Tristram-Levine signature}), while $\sigma_L(-1)$ is referred to as \emph{the} signature of $L$ or as the \emph{Murasugi signature} of $L$~\cite{Murasugi}. The definition of $\sigma_L(\omega)$ goes back to Tristram~\cite{Tristram} and Levine~\cite{LevineKnotCob}.

\begin{remark}
\label{rem:LocallyConstant}
We argue that $\sigma_L$ and $\eta_L$ are piecewise constant. Both observations follow from the fact that the Alexander polynomial  $\Delta_L(t)$ can be computed (up to its indeterminacy) by the formula $\Delta_L(t)=\det(tA-A^T)$. Thus, given $\omega \in S^1 \setminus \lbrace 1 \rbrace $, the signature $\sigma_L \colon S^1 \rightarrow \Z$ is piecewise constant, and the nullity $\eta_L(\omega)$ vanishes if and only if~$\Delta_L(\omega) \neq 0$. Moreover, the discontinuities of $\sigma_L$ only occur at zeros of $(t-1)\Delta_L^{\mathrm{tor}}(t)$ \cite[Theorem 2.1]{GilmerLivingston}. 
\end{remark}

Several authors assume Seifert surfaces to be connected, and the nullity is then simply defined as the nullity of the matrix $H(\omega)=(1-\omega)A+(1-\overline{\omega})A^T$. The extra flexibility afforded by disconnected Seifert surfaces can for instance be taken advantage of when studying the behavior of the signature and nullity of boundary links.

\begin{remark}
\label{rem:Not1}
Since the matrix $(1-\omega)A+(1-\overline{\omega})A^T$ vanishes at $\omega=1$, we shall frequently think of $\sigma_L$ and $\eta_L$ as functions on $S^1_*:=S^1 \setminus \lbrace 1 \rbrace$. Note nevertheless that for a knot $K$, the function~$\sigma_K$ vanishes in a neighborhood of $1 \in S^1$~\cite[page 255]{LevineMetabolicHyperbolic}, while for a $\mu$-component link, one can only conclude that the limits of $|\sigma_L(\omega)|$ are at most $\mu-1$ as $\omega$ approaches $1$. 
%Might be a jump at $1$.
%For further remarks on the value of $\sigma_L$ at $1$, we refer to~\cite{DegtyarevFlorensLecuona}.
\end{remark}

\subsection{Properties of the signature and nullity}
\label{sub:ElemProperties}
This subsection discusses the behaviour of the signature and nullity under operations such as orientation reversal, mirror image, connected sums and satellite operations.
\medbreak
The following proposition collects several properties of the Levine-Tristram signature.

\begin{proposition}
\label{prop:PropertiesSignature}
Let $L$ be a $\mu$-component oriented link and let $\omega \in S^1$.
\begin{enumerate}
\item The Levine-Tristram signature is symmetric: $\sigma_L(\overline{\omega})=\sigma_L(\omega)$.
\item The integer $\sigma_L(\omega)+\eta_L(\omega)-\mu+1$ is even.
\item If $\Delta_L(\omega) \neq 0$, then $ \sigma_L(\omega)=\mu-\operatorname{sgn}(i^\mu \nabla_L(\sqrt{\omega}))$ mod~$4$. \footnote{Here $\nabla_L(t)$ denotes the \emph{one variable potential function} of $L$. Given a Seifert matrix $A$ for $L$, the \emph{normalized Alexander polynomial} is $D_L(t)=\det(-tA+t^{-1}A^T)$ and $\nabla_L(t)$ can be defined as $\nabla_L(t)=D_L(t)/(t-t^{-1})$.}
\item If $L^*$ denotes the mirror image of $L$, then $\sigma_{L^*}(\omega)=-\sigma_L(\omega)$.
\item If $-L$ is obtained by reversing the orientation of each component of $L$, then $\sigma_{-L}(\omega)=~\sigma_L(\omega)$.
\item Let $L'$ and $L''$ be two oriented links. If $L$ is obtained by performing a single connected sum between a component of $L'$ and a component of $L''$, then $\sigma_{L}(\omega)=\sigma_{L'}(\omega)+\sigma_{L''}(\omega)$.
\item The signature is additive under the disjoint sum operation: if $L$ is the link obtained by taking the disjoint union of two oriented links $L'$ and $L''$, then $\sigma_{L}(\omega)=\sigma_{L'}(\omega)+\sigma_{L''}(\omega)$.
\item If $S$ is a satellite knot with companion knot $C$, pattern $P$ and winding number $n$, then 
$$ \sigma_S(\omega)=\sigma_P(\omega)+\sigma_C(\omega^n). $$
\end{enumerate}
\end{proposition}
\begin{proof}
The first assertion is immediate from Definition~\ref{def:LevineTristram}. The proof of the second and third assertions can be found respectively in~\cite{Shinohara}; see also \cite[Lemmas 5.6 and 5.7]{CimasoniFlorens}. The proof of the third assertion can be found in~\cite[Theorem 8.10]{Lickorish}; see also~\cite[Proposition~2.10]{CimasoniFlorens}. The proof of the fifth, sixth and seventh assertions can be respectively be found in~\cite[Corollary 2.9, Proposition~2.12, Proposition 2.13]{CimasoniFlorens}. For the proof of the last assertion, we refer to~\cite[Theorem~2]{LitherlandIterated}; see also~\cite[Theorem 9]{Shinohara} (and~\cite[Theorem 3]{LivingstonMelvin}) for the case $\omega=-1$.
\end{proof}

Note that the second and third assertions of Proposition~\ref{prop:PropertiesSignature} generalize the well known fact that the Murasugi signature of a knot $K$ is even. 
The behavior of $\sigma_L$ under splicing (a generalization of the satellite operation) is discussed in~\cite{DegtyarevFlorensLecuona, DegtyarevFlorensLecuona2}.
For discussions on the (Murasugi) signature of covering links, we refer to~\cite{Murasugi,GordonLitherlandTheoremOfMurasugi} and~\cite{GordonLitherlandMurasugi} (which also provides a signature obstruction to a knot being periodic).

The following proposition collects the corresponding properties of the nullity.

\begin{proposition}
\label{prop:PropertiesNullity}
Let $L$ be an oriented link and let $\omega \in S^1_*:=S^1 \setminus \lbrace 1 \rbrace$.
\begin{enumerate}
\item The nullity is symmetric: $\eta_L(\overline{\omega})=\eta_L(\omega)$.
\item The nullity $\eta_L(\omega)$ is nonzero if and only if $\Delta_L(\omega) = 0$.
\item If $L^*$ denotes the mirror image of $L$, then $\eta_{L^*}(\omega)=\eta_L(\omega)$.
\item If $-L$ is obtained by reversing the orientation of each component of $L$, then $\eta_{-L}(\omega)=~\eta_L(\omega)$.
\item Let $L'$ and $L''$ be two oriented links. If $L$ is obtained by performing a single connected sum between a component of $L'$ and a component of $L''$, then $\eta_L(\omega)=\eta_{L'}(\omega)+\eta_{L''}(\omega)$.
\item If $L$ is the link obtained by taking the disjoint union of two oriented links $L'$ and $L''$, then we have $\eta_{L}(\omega)=\eta_{L'}(\omega)+\eta_{L''}(\omega)+1$.
\item The nullity $\eta_L(\omega)$ is equal to the dimension of the twisted homology $\C$-vector space $H_1(X_L;\C_\omega)$, where $\C_\omega$ is the right $\Z[\pi_1(X_L)]$-module arising from the map $\Z[\pi_1(X_L)] \to \C, \gamma \to~\omega^{\ell k(\gamma, L)}$.
\item If $S$ is a satellite knot with companion knot $C$, pattern $P$ and winding number $n$, then 
$$ \eta_S(\omega)=\eta_P(\omega)+\eta_C(\omega^n). $$
\end{enumerate}
\end{proposition}
\begin{proof}
The first assertion is immediate from Definition~\ref{def:LevineTristram}, while the second assertion was already discussed in Remark~\ref{rem:LocallyConstant}. The proof of assertions $(3)-(6)$ can respectively be found in~\cite[Proposition 2.10, Corollary 2.9, Proposition 2.12, Proposition 2.13]{CimasoniFlorens}.
To prove the penultimate assertion, pick a connected Seifert surface $F$ for $L$, let $A$ be an associated Seifert matrix and set $H(\omega)=(1-\omega)A+(1-\overline{\omega})A^T$. Since $tA-A^T$ presents the Alexander module $H_1(X_L;\Z[t^{\pm 1}])$, some homological algebra shows that $H(\omega)$ presents $H_1(X_L;\C_\omega)$; the assertion follows.
The satellite formula can be deduced from~\cite[Theorem 5.2]{DegtyarevFlorensLecuona2}, or by using the equality~$\eta_S(\omega)=\dim_\C H_1(X_S;\C^\omega)$ and running a Mayer-Vietoris argument for $H_1(X_S;\C^\omega)$.
\end{proof}

We conclude this subsection by mentioning some additional facts about the signature function. Livingston provided a complete characterization the functions $\sigma \colon S^1 \to \Z$ that arise as the Levine-Tristram signature function of a knot~\cite{LivingstonCharac}. The corresponding question for links appears to be open.  If $\Delta_L(t)$ is not identically zero, then it has at least $\sigma(L)$ roots on the unit circle~\cite[Appendix]{Liechti}. Finally, we describe the Murasugi signature for some particular classes of links.

\begin{remark}
\label{rem:Positive}
Rudolph showed that the Murasugi signature of the closure of a nontrivial positive braid is negative (or positive, according to conventions)~\cite{RudolphPositive}. This result was later independently extended to positive links~\cite{Traczyk, Przytycki} (see also~\cite{CochranGompf}) and to almost positive links~\cite{PrzytyckiTaniyama}. Later, Stoimenow improved Rudolph's result by showing that the Murasugi signature is bounded by an increasing function of the first Betti number~\cite{StoimenowPositive}. Subsequent improvements of this result include~\cite{FellerLinearBound, BaaderDehornoyLiechti}. Formulas for the Levine-Tristram signature of torus knots can be found in~\cite{LitherlandIterated}.
%Note also that $\sigma_L(\omega) \leq 0$ if $L$ is a positive braid knot; this can be seen using~\ref{prop:CrossingChange}.
%Combine the first AND the third point.
\end{remark}

%The behavior of the Levine-Tristram signature under crossing changes is well understood and can expressed using the potential function~ \cite[Section 5]{CimasoniFlorens} and~\cite{GaroufalidisSignature}. 

\subsection{Lower bounds on the unlinking number}
\label{sub:Unlinking}

In this subsection, we review some applications of signatures to unlinking and splitting links.
\medbreak
The \emph{unlinking number} $u(L)$ of a link $L$ is the minimal number of crossing changes needed to turn~$L$ into an unlink. The \emph{splitting number} $\operatorname{sp}(L)$ of $L$ is the minimal number of crossing changes between different components needed to turn $L$ into the split union of its components. The Levine-Tristram signature and nullity are known to provide lower bounds on both these quantities: 

\begin{theorem}
\label{thm:Unlinking}
Let  $L=L_1\cup \ldots \cup L_\mu$ be an oriented link and let $\omega \in S^1_*=S^1 \setminus \lbrace 1 \rbrace$.
\begin{enumerate}
\item The signature provides lower bounds on the unlinking number:
$$ |\sigma_L(\omega)|+|\eta_L(\omega)+\mu-1| \leq 2u(L).$$
\item The signature and nullity provide lower bounds on the splitting number:
$$ \Big|\sigma_L(\omega)+\sum_{i<j}\ell k(L_i,L_j)-\sum_{i=1}^\mu \sigma_{L_i}(\omega)\Big|+\Big|\mu-1-\eta_L(\omega)+\sum_{i=1}^\mu \eta_{L_i}(\omega)\Big|\leq\operatorname{sp}(L).$$
\end{enumerate}
\end{theorem}

At the time of writing, the second inequality can only be proved using the multivariable signature~\cite{CimasoniConwayZacharova}. A key step in proving the first inequality is to understand the behavior of the signature and nullity under crossing changes. The next proposition collects several such results:

\begin{proposition}
\label{prop:CrossingChange}
Given, $\omega \in S^1_*$, the following assertions hold.
\begin{enumerate}
\item If $L_+$ is obtained from $L_-$ by changing a single negative crossing change, then
$$(\sigma_{L_+}(\omega) \pm  \eta_{L_+}(\omega))-(\sigma_{L_-}(\omega)\pm\eta_{L_-}(\omega)) \in \lbrace 0,-2 \rbrace .$$
\item If, additionally, we let $\mu$ denote the number of components of $L_+$ (and $L_-$) and assume that $\omega$ is neither a root of $\Delta_{L_-}(t)$ nor of $\Delta_{L_+}(t)$, then
$$ \sigma_{L_+}(\omega)-\sigma_{L_-}(\omega)=
\begin{cases}
0              & \mbox{if }  (-1)^\mu\nabla_{L_+}(\sqrt{\omega})\nabla_{L_-}(\sqrt{\omega}) >0, \\
-2              & \mbox{if }  (-1)^\mu \nabla_{L_+}(\sqrt{\omega})\nabla_{L_-}(\sqrt{\omega}) <0.
 \end{cases} 
$$
\item If $L$ and $L'$ differ by a single crossing change, then
$$ |\eta_L(\omega)-\eta_{L'}(\omega)| \leq 1.  $$
\end{enumerate}
\end{proposition}
\begin{proof}
The proof of the first and third assertions can be found in~\cite[Lemma 2.1]{NagelOwens} (the proof is written for $\omega=-1$, but also holds for general $\omega$). The proof of the second assertion now follows from the second item of Proposition~\ref{prop:PropertiesSignature} which states that modulo $4$, the signature $\sigma_L(\omega)$ is congruent to~$\mu +1$ or $\mu-1$ according to the sign of $i^\mu \nabla_L(\sqrt{\omega})$.
\end{proof}

Note that similar conclusions hold if $L_-$ is obtained from $L_+$ by changing a single negative crossing change; we refer to~\cite[Lemma 2.1]{NagelOwens} for the precise statement.
Although Proposition~\ref{prop:CrossingChange} is well known, it seems that the full statement is hard to find in the literature: subsets of the statement for knots appear for instance in~\cite{FellerGordian, Garoufalidis, HeusenerKroll} (away from the roots of $\Delta_K(t)$) and for links in~\cite[Lemma 2.1]{NagelOwens} (for~$\omega=-1$, without the statement involving $\nabla_L$), and~\cite[Section 5]{CimasoniFlorens} (in which various local relations are described; see also~\cite[Section 7.10]{ConwayJohn} and~\cite[Lemma 3.1]{OrevkovNest}).

We conclude this subsection with two additional remarks in the knot case.

\begin{remark}
\label{rem:KnotsUnknotting}
In the knot case, the second assertion of Proposition~\ref{prop:CrossingChange} is fairly well known (e.g.~\cite[Lemma 2.2]{Garoufalidis} and~\cite[Equation (10)]{HeusenerKroll}). Indeed, under the same assumptions as in Proposition~\ref{prop:CrossingChange}, and using the normalized Alexander polynomial $D_L(t)$ (which satisfies $D_L(t)=(t-t^{-1})\nabla_L(t)$), it can be rewritten as
$$ \sigma_{K_+}(\omega)-\sigma_{K_-}(\omega)=
\begin{cases}
0              & \mbox{if }  D_{K_+}(\sqrt{\omega})D_{K_-}(\sqrt{\omega}) >0, \\
-2              & \mbox{if }   D_{K_+}(\sqrt{\omega})D_{K_-}(\sqrt{\omega}) <0.
 \end{cases}
$$
Finally, note that for knots, the lower bound on the unknotting number can be significantly improved upon by using the jumps of the signature function~\cite{LivingstonUnknotting}. Other applications of the Levine-Tristram signature to unknotting numbers can be found in~\cite{StoimenowApplications} (as well as a relation to finite type invariants).
\end{remark}

%Also a word on the splitting number. Various more subtle remarks in Borodzik-Powell for unknotting.
\subsection{Concordance invariance and the Murasugi-Tristram inequalities}
\label{sub:Concordance}

In this subsection, we review properties of the Levine-Tristram signature related to $4$-dimensional topology. Namely we discuss the conditions under which the signature is a concordance invariant, and lower bounds on the $4$-genus. 

\medbreak

Two oriented $\mu$-component links $L$ and $J$ are smoothly (resp. topologically) \emph{concordant} if there is a smooth (resp. locally flat) embedding into $S^3 \times I$ of a disjoint union of $\mu$ annuli $A \hookrightarrow S^3 \times I$, such that the oriented boundary of $A$ satisfies 
$$\partial A = -L \sqcup J \subset  -S^3 \sqcup S^3 = \partial(S^3 \times I).$$
The integers $\sigma_L(\omega)$ and $\eta_L(\omega)$ are known to be concordance invariants for any root of unity~$\omega$ of prime power order~\cite{Murasugi, Tristram}. However, it is only recently that Nagel and Powell gave a precise characterization of the $\omega \in S^1$ at which $\sigma_L$ and $\eta_L$ are concordance invariants~\cite{NagelPowell} (see also~\cite{Viro09}). To describe this characterization, we say that a complex number $\omega \in S^1_*$ is a \emph{Knotennullstelle} if it is the root of a Laurent polynomial $p(t) \in \Z[t^{\pm 1}]$ satisfying $p(1)=\pm 1$. We write $S^1_!$ for the set of $\omega \in S^1$ which do \emph{not} arise a Knotennullstelle.

The main result of~\cite{NagelPowell} can be stated as follows.
\begin{theorem}
\label{thm:NagelPowell}
The Levine-Tristram signature $\sigma_L$ and nullity $\eta_L$ are concordance invariants at $\omega \in S^1_*$ if and only if $\omega \in S^1_!$.
\end{theorem}

In the knot case, Cha and Livingston had previously shown that for any Knotennullstelle $\omega$, there exists a slice knot $K$ with $\sigma_K(\omega) \neq 0$ and $\eta_K(\omega) \neq 0$~\cite{ChaLivingston}. Here, recall that a knot $K \subset S^3$ is smoothly (resp. topologically) \emph{slice} if it is smoothly (resp. topologically) concordant to the unknot or, equivalently, if it bounds a smoothly (resp. locally flat) properly embedded disk in the $4$-ball. Still restricting to knots, the converse can be established as follows.

\begin{remark}
\label{rem:AlgebraicConcordance}
The Levine-tristram signature of an oriented knot $K$ vanishes at $\omega \in S^1_!$ whenever~$K$ is algebraically slice i.e. whenever it admits a metabolic Seifert matrix $A$. To see this, first note that since $A$ is metabolic, the matrix $H(\omega)=(1-\omega)A+(1-\overline{\omega})A^T$ is congruent to one which admits a half size block of zeros in its upper left corner. Furthermore the definition of~$S^1_!$ and the equality~$H(t)=(t^{-1}-1)(tA-A^T)$ imply that $H(\omega)$ is nonsingular for $\omega \in S^1_!$: indeed, since~$K$ is a knot, $\Delta_K(1)=\pm 1$. Combining these facts, $\sigma_K(\omega)=\operatorname{sign}(H(\omega))$ vanishes for~$\omega \in S^1_!$. As slice knots are algebraically slice (see e.g.~\cite[Proposition 8.17]{Lickorish}), we have established that if~$K$ is slice, then $\sigma_K$ vanishes on $S^1_!$.
\end{remark}

Using Remark~\ref{rem:AlgebraicConcordance} and Theorem~\ref{thm:Unlinking}, one sees that the Levine-Tristram signature actually provides lower bounds on the \emph{slicing number} of a knot $K$ i.e. the minimum number of crossing changes required to convert $K$ to a slice knot~\cite{LivingstonSlicingNumber, OwensSlicingNumber}. In a somewhat different direction, the Levine-Tristram signature is also a lower bound on the algebraic unknotting number~\cite{Fogel,MurakamiUnknotting, BorodzikFriedl0, BorodzikFriedl, BorodzikFriedl2}.

Several steps in Remark~\ref{rem:AlgebraicConcordance} fail to generalize from knots to links: there is no obvious notion of algebraic sliceness for links and, if $L$ has two components or more, then $\Delta_L(1)=0$. In fact, even the notion of a slice link deserves some comments.

\begin{remark}
\label{rem:SliceLink}
An oriented link $L=L_1 \cup \ldots \cup L_\mu$ is smoothly (resp. topologically) \emph{slice in the strong sense} if there are disjointly smoothly (resp. locally flat) properly embedded disks $D_1,\ldots,D_\mu$ with $\partial D_i=L_i$. As a corollary of Theorem~\ref{thm:NagelPowell}, one sees that if $L$ is topologically slice in the strong sense, then $\sigma_L(\omega)=0$ and $\eta_L(\omega)=\mu-1$ for all $\omega \in S^1_!$.

On the other hand, an oriented link is smoothly (resp. topologically) \emph{slice in the ordinary sense} if it is the cross-section of a single smooth (resp. locally flat) $2$-sphere in $S^4$. It is known that if~$L$ is slice in the ordinary sense, then $\sigma_L(\omega)=0$ for all $\omega$ of prime power order~\cite[Corollary 7.5]{CimasoniFlorens} (see also~\cite[Theorem 3.13]{KauffmanTaylor}). There is little doubt that this result should hold for a larger subset of~$S^1$ and in the topological category.
\end{remark}

In a similar spirit, the Levine-Tristram signatures can be used to provide restrictions on the surfaces a link can bound in the 4-ball. Such inequalities go back to Murasugi~\cite{Murasugi} and Tristram~\cite{Tristram}. Since then, these inequalities have been generalized in several directions~\cite[Corollary 4.3]{GilmerConfiguration},~\cite[Theorem 5.19]{Florens},~\cite[Theorem 7.2]{CimasoniFlorens},~\cite{LivingstonWitt},~\cite[Section 4]{Viro09} and~\cite[Theorem 1.2 and Corollary 1.4]{ConwayNagelToffoli}. 
Applications to the study of algebraic curves can be found in~\cite{Florens, OrevkovNest, OrevkovSomeExampleMultivariable}.

The following theorem describes such a \emph{Murasugi-Tristram inequality} in the topological category which holds for a large subset of $S^1$.

\begin{theorem}
\label{thm:MurasugiTristram}
If an oriented link $L$ bounds an $m$-component properly embedded locally flat surface $F \subset D^4$ with first Betti number $b_1(F)$, then for any $\omega \in S^1_!$, the following inequality holds:
$$ |\sigma_L(\omega)|+|\eta_L(\omega)-m+1| \leq b_1(F).$$
\end{theorem}

Observe that if $L$ is a strongly slice link, then $m$ is equal to the number of components of~$L$ and~$b_1(F)=~0$ and thus $\sigma_L(\omega)=0$ and $\eta_L(\omega)=m-1$ for all $\omega \in S^1_!$, recovering the result mentioned in Remark~\ref{rem:SliceLink}. On the other hand, if $K$ is a knot, then Theorem~\ref{thm:MurasugiTristram} can be expressed in terms of the topological $4$-genus~$g_4(K)$ of $K$: the minimal genus of a locally flat surface in $D^4$ cobounding $K$. An article studying the sharpness of this bound includes~\cite{LiechtiPositiveBraid}.
%~\cite{BaaderMaximal}.

In order to obtain results which are valid on the whole of $S^1$, it is possible to consider the average of the one-sided limits of the signature and nullity. Namely for $\omega=e^{i\theta} \in S^1$ and any Seifert matrix $A$, one sets $H(\omega)=(1-\omega)A+(1-\overline{\omega})A^T$ and considers
\begin{align*}
&\sigma_L^{\operatorname{av}}(\omega)=\frac{1}{2}\big( \lim_{\eta \to \theta_+} \operatorname{sign}(H(e^{i\eta})) +\lim_{\eta \to \theta_-} \operatorname{sign}(H(e^{i\eta})) \big), \\
&\eta^{\operatorname{av}}(\omega)=\frac{1}{2}\big( \lim_{\eta \to \theta_+} \operatorname{null}(H(e^{i\eta})) +\lim_{\eta \to \theta_-} \operatorname{null}(H(e^{i\eta})) \big).
\end{align*}
The earliest explicit observation that these \emph{averaged Levine-Tristram signatures} are smooth concordance invariants seems to go back to Gordon's survey~\cite{GordonSurvey}. Working with the averaged Levine-Tristram signature and in the topological locally flat category, Powell~\cite{PowellSignature} recently proved a Murasugi-Tristram type inequality which holds for \emph{each} $\omega \in S^1_*$.

We conclude this subsection with two remarks on knots.
\begin{remark}
\label{rem:DoublySlice}
A knot $K$ is smoothly (resp. topologically) \emph{doubly slice} if it is the cross section of an unknotted smoothly (resp. locally flat) embedded 2-sphere $S^2$ in $S^4$. It is known that if $K$ is topologically doubly slice, then $\sigma_K(\omega)$ vanishes for \emph{all}~$\omega \in S^1$; no averaging is needed~\cite{Sumners, KeartonDouble, LevineMetabolicHyperbolic}. Is there a meaningful statement for links?

The Levine-Tristram signature also appears in knot concordance in relation to a particular von Neumann $\rho$-invariant (or $L^2$-signature). This invariant associates a real number to any closed $3$-manifold together with a map $\phi \colon \pi_1(M) \to \Gamma$, with $\Gamma$ a PTFA group. When $M$ is the $0$-framed surgery along a knot $K$ and~$\phi$ is the abelianization map, then this invariant coincides with the (normalized) integral of~$\sigma_K(\omega)$ along the circle~\cite[Proposition 5.1]{CochranOrrTeichnerStructure}. Computations of this invariant on (iterated) torus knots can be found in~\cite{KirbyMelvin, BorodzikIntegral,Collins}.
\end{remark}

\section{4-dimensional definitions of the signature}
\label{sec:4DDefinitions}

In this section, we describe 4-dimensional definitions of the Levine-Tristram signature using embedded surfaces in the 4-ball (Subsection~\ref{sub:4DLevineTristramSurfaces}) and as a bordism invariant of the 0-framed surgery (Subsection~\ref{sub:4DSignatureBordismDef}).
%To the best of our knowledge, 4-dimensional definitions of the signature are not featured in knot theory textbooks, with the exception of~\cite{KauffmanOnKnots}. 

\subsection{Signatures via exteriors of surfaces in the $4$-ball}
\label{sub:4DLevineTristramSurfaces}
We relate the Levine-Tristram signature to signature invariants of the exterior of embedded surfaces in the 4-ball. Historically, the first approach of this kind involved branched covers~\cite{Viro} (see also~\cite{CassonGordonSlice, KauffmanTaylor}) while more recent results make use of twisted homology~\cite{CochranOrrTeichnerStructure, Viro09, PowellSignature}.  
\medbreak
Given a smoothly properly embedded connected surface~$F \subset D^4$, denote by $W_F$ the complement of a tubular neighborhood of~$F$. A short Mayer-Vietoris argument shows that $H_1(W_F;\Z)$ is infinite cyclic and one may consider the covering space $W_k \to W_F$ obtained by composing the abelianization homomorphism with the quotient map $H_1(W_F;\Z)\cong \Z \to \Z_k$. The restriction of this cover to $F \times S^1$ consists of a copy of $F \times p^{-1}(S^1)$, where $p \colon S^1\rightarrow S^1$ is the $k$-fold cover of the circle. Extending $p$ to a cover $D^2 \rightarrow D^2$ branched along $0$, and setting
$$ \overline{W}_F :=W_k \cup_{F \times S^1} (F \times D^2)$$
produces a cover $\overline{W}_F \rightarrow D^4$ branched along $F=F \times \lbrace 0 \rbrace$. Denote by $t$ a generator of the finite cyclic group $\Z_k$. The $\C[\Z_k]$-module structure of $H_2(\overline{W}_F,\C)$ gives rise to a complex vector space
$$ H_2(\overline{W}_F,\C)_\omega = \lbrace x \in H_2(\overline{W}_F,\C) \ | \ tx=\omega x \rbrace $$
for each root of unity $\omega$ of order $k$. Restricting the intersection form on $H_2(\overline{W}_F,\C)$ to $H_2(\overline{W}_F,\C)_\omega$ produces a Hermitian pairing whose signature we denote by $\sigma_\omega(\overline{W}_F)$. 
%These signatures share most of the properties of the usual signature~\cite[Chapter XIII]{KauffmanOnKnots}. Although this fact is not altogether surprising (since we are merely restricting to summands of $H_2(\overline{W}_F;\C)$), it can also be understood in the somewhat deeper context of G-signatures~\cite{KauffmanOnKnots, CassonGordonSlice}.

The next result, originally due to Viro~\cite{Viro}, was historically the first 4-dimensional interpretation of the Levine-Tristram signature; see also~\cite{KauffmanTaylor}. 

\begin{theorem}
\label{thm:LevineTristram4D}
Assume that an oriented link $L$ bounds a smoothly properly embedded compact oriented surface $F \subset D^4$ and let $\overline{W}_F$ be the $k$-fold cover of $D^4$ branched along~$F$. Then, for any root of unity $\omega \in S^1_*$ of order $k$, the following equality holds:
$$ \sigma_L(\omega)=\sigma_\omega(\overline{W}_F).$$
\end{theorem}

As for the results described in Subsection~\ref{sub:Concordance}, Theorem~\ref{thm:LevineTristram4D} can be sharpened by working in the topological category and using arbitrary $\omega \in S^1_*$. 
As the next paragraphs detail, the idea is to rely on twisted homology instead of branched covers~\cite{PowellSignature, Viro09, ConwayNagelToffoli}.

Let $\omega \in S^1_*$. From now on, we assume that $F \subset D^4$ is a locally flat properly embedded (possibly disconnected) compact oriented surface. Since $H_1(W_F;\Z)$ is free abelian, there is a map~$H_1(W_F;\Z) \to \C$ obtained by sending each meridian of $F$ to $\omega$. Precomposing with the abelianization homomorphism, gives rise to a right $\Z[\pi_1(W_F)]$-module structure on $\C$ which we denote by $\C_\omega$ for emphasis. We can therefore consider the twisted homology groups $H_*(W_F;\C_\omega)$ and the corresponding $\C$-valued intersection form $\lambda_{W_F,\C_\omega}$ on $H_2(W_F;\C_\omega)$.

The following result can be seen as a generalization of Theorem~\ref{thm:LevineTristram4D}.

\begin{theorem}
\label{thm:4DDefSurfaceTop}
Assume that an oriented link $L$ bounds a properly embedded locally flat compact oriented surface $F \subset D^4$. Then the following equality holds for any $\omega \in S^1_*$:
$$ \sigma_L(\omega)=\operatorname{sign}(\lambda_{W_F,\C_\omega}).$$
\end{theorem}

A key feature of Theorems~\ref{thm:LevineTristram4D} and~\ref{thm:4DDefSurfaceTop} lies in the fact that the signature invariants associated to $W_F$ do not depend on the choice of $F$. This plays a crucial role in the $4$-dimensional proofs of Murasugi-Tristram type inequalities.
This independance statement relies on the Novikov-Wall addivity as well as on the G-signature theorem (for Theorem~\ref{thm:LevineTristram4D}) and on bordisms considerations over the classifying space $B\Z$ (for Theorem~\ref{thm:4DDefSurfaceTop}).

\subsection{Signatures as invariants of the $0$-framed surgery}
\label{sub:4DSignatureBordismDef}

In this subsection, we outline how the Levine-Tristram signature of a link $L$ can be viewed as a bordism invariant of the 0-framed surgery along $L$. To achieve this, we describe bordism invariants of pairs consisting of a closed connected oriented 3-manifold together with a map from $\pi_1(M)$ to $\Z_m$ or $\Z$.
\medbreak

Let $M$ be an oriented closed $3$-manifold and let $\chi \colon \pi_1(M) \to \Z_m$ be a homomorphism. Since the bordism group $\Omega_3(\Z_m)$ is finite, there exists a non-negative integer~$r$, a $4$-manifold $W$ and a map $\psi \colon \pi_1(W) \to \Z_m$ such that the boundary of $W$ consists of the disjoint union of $r$ copies of $M$ and the restriction of $\psi$ to $\partial W$ coincides with $\chi$ on each copy of $M$. If these conditions are satisfied, we write $\partial (W,\psi)=r(M,\chi)$ for brevity. Mapping the generator of $\Z_m$ to $\omega:=e^{\frac{2 \pi i}{m}}$ gives rise to a map $\Z[\Z_m] \to \Q(\omega)$. Precomposing with $\psi$, we obtain a $(\Q(\omega),\Z[\pi_1(W)])$-bimodule structure on $\Q(\omega)$ and twisted homology groups $H_*(W;\Q(\omega))$. The $\Q(\omega)$-vector space $H_2(W;\Q(\omega))$ is endowed with a $\Q(\omega)$-valued Hermitian form~$\lambda_{W,\Q(\omega)}$ whose signature is denoted $ \text{sign}^\psi(W):=\text{sign}(\lambda_{W,\Q(\omega)}).$ In this setting, the \emph{Casson-Gordon $\sigma$-invariant} of $(M,\chi)$ is
$$\sigma(M,\chi):=\frac{1}{r}\left (\text{sign}^\psi(W)-\text{sign}(W)\right ) \in \mathbb{Q}.$$
We now focus on the case where $M=M_L$ is the closed $3$-manifold obtained by performing $0$-framed surgery on a link $L$. In this case, a short Mayer-Vietoris argument shows that $H_1(M_L;\Z)$ is freely generated by the meridians of $L$.

Casson and Gordon proved the following theorem~\cite[Lemma 3.1]{CassonGordonSlice}.

\begin{theorem}
\label{thm:CG}
Assume that $m$ is a prime power and let $\chi \colon H_1 (M_L;\Z) \to \Z_m \subset \C^*$ be the character mapping each meridian of $L$ to $\omega^r$, where $\omega=e^{\frac{2 \pi i}{m}}$ and $0 < r < m$. Then the Casson-Gordon $\sigma$-invariant satisfies
$$ \sigma(M_L,\chi)=\sigma_L(\omega^r).$$
\end{theorem}
%Mapping the meridian to an arbitrary $\omega$ should give the same result; the point is to use $\C$-coefficients and then Novikov-Wall additivity on that level using the previous theorem (this is basically Powell)

Note that Casson and Gordon proved a version of Theorem~\ref{thm:CG} for arbitrary surgeries on links; we also refer to~\cite[Theorem 3.6]{GilmerConfiguration} and~\cite[Theorem 6.7]{CimasoniFlorens} for generalizations to more general characters. The idea of defining link invariants using the Casson-Gordon invariants is pursued further in~\cite{Florens, FlorensGilmer}.

\begin{remark}
\label{rem:Eta}
The Casson-Gordon $\sigma$-invariant (and thus the Levine-Tristram signature) can be understood as a particular case of the Atiyah-Patodi-Singer $\rho$-invariant~\cite{AtiyahPatodiSinger} which associates a real number to pairs $(M,\alpha)$, with $M$ a closed connected oriented 3-manifold and $\alpha \colon \pi_1(M) \to U(k)$ a unitary representation. For further reading on this point of view, we refer to~\cite{LevineEta, Letsche, FriedlEtaCasson, FriedlEtaL2, FriedlEtaLink}.
\end{remark}

Next, we describe how to circumvent the restriction that $\omega$ be of finite order. Briefly, the idea is to work in the infinite cyclic cover as long as possible, delaying the appearance of $\omega$~\cite[Section~2]{LitherlandCobordism}; see also~\cite[Section 5]{CochranOrrTeichnerStructure}. Following~\cite{PowellSignature}, the next paragraphs describe the resulting construction.

Let $M$ be a closed connected oriented $3$-manifold with a map $\phi \colon \pi_1(M) \to \Z$. Since $\Omega_3^{STOP}(\Z)$ is zero, $M$ bounds a connected topological $4$-manifold $W$ and there is a map $\psi \colon \pi_1(W) \to \Z$ which extends $\phi$. This map endows $\Q(t)$ with a $(\Q(t),\Z[\pi_1(W)])$-bimodule structure and therefore gives rise to a $\Q(t)$-valued intersection form~$\lambda_{W,\Q(t)}$ on $H_2(W;\Q(t))$. It can be checked that~$\lambda_{W,\Q(t)}$ induces a nonsingular Hermitian form~$\lambda_{W,\Q(t)}^{\operatorname{nonsing}}$ on the quotient of $H_2(W;\Q(t))$ by $\operatorname{im} (H_2(M;\Q(t)) \to H_2(W;\Q(t)))$~\cite[Lemma 3.1]{PowellSignature}. As a consequence,~$\lambda_{W,\Q(t)}^{\operatorname{nonsing}}$ gives rise to an element $[\lambda_{W,\Q(t)}^{\operatorname{nonsing}}]$ of the Witt group $W(\Q(t))$. Taking the averaged signature at $\omega \in S^1$ of a representative of an element in $W(\Q(t)) $ produces a well defined homomorphism $\operatorname{sign}_\omega \colon W(\Q(t)) \to \C$. As a consequence, for $\omega \in S^1_*$ and $(M,\phi)=\partial (W,\psi)$ as above, one can set
$$ \sigma^{\operatorname{av}}_{M,\phi}(\omega)=\operatorname{sign}_\omega([\lambda_{W,\Q(t)}^{\operatorname{nonsing}}])-\operatorname{sign}(W). $$ 
%Mark claims that this lies in $\Z$. I don't see why.
It can be checked that $\sigma_{M,\phi}^{\operatorname{av}}$ does not depend on $W$ and $\psi$~\cite[Section 3]{PowellSignature}. We now return to links: we let $L$ be an oriented link, assume that $M$ is the 0-framed surgery $M_L$ and that $\phi$ is the map~$\phi_L$ which sends each meridian of $L$ to $1$. 

The following result is due to Powell~\cite[Lemma 4.1]{PowellSignature}.

\begin{theorem}
\label{thm:Powell}
For any oriented link $L$ and any $\omega \in S^1_*$, the following equality holds:
$$\sigma_{M_L,\phi_L}^{\operatorname{av}}(\omega)=\sigma_L^{\operatorname{av}}(\omega).$$
\end{theorem}

Theorem~\ref{thm:Powell} has two main strengths. Firstly, it holds for all $\omega \in S^1_*$. Secondly, thanks to the definition of~$\sigma^{\operatorname{av}}_{M,\phi}(z)$, it provides a useful tool to work in the topological category (see e.g Powell's proof a Murasugi-Tristram type inequality~\cite[Theorem 1.4]{PowellSignature}).

\section{Signatures via pairings on infinite cyclic covers}
\label{sec:3DDefOther}

In this section, we review two additional intrinsic descriptions of the Levine-Tristram signature of a knot $K$. Both constructions make heavy use of the algebraic topology of the infinite cyclic cover of the exterior of $K$: the first uses the Milnor pairing (Subsection~\ref{sub:Milnor}), while the second relies on the Blanchfield pairing (Subsection~\ref{sub:Blanchfield}). 

\subsection{Milnor signatures}
\label{sub:Milnor}

%jumps and integrals; cite Borodzik, Collins etc.
%
In this subsection, we recall the definition of a pairing which was first described by Milnor~\cite{MilnorInfiniteCyclic}. We then outline how the resulting ``Milnor signatures" are related to (the jumps of) the Levine-Tristram signature. 
\medbreak

%Given an oriented knot $K$, recall that $X_K^\infty$ denotes the infinite cyclic cover of the exterior $X_K$. 
Given an oriented knot~$K$ in $S^3$, use $X_K=S^3 \setminus \nu K$ to denote its exterior. The kernel of the abelianization homomorphism $\pi_1(X_K) \to H_1(X_K;\Z)\cong \Z$ gives rise to an infinite cyclic cover~$X_K^\infty \to X_K$.
Milnor showed that the cup product 
$$H^1(X_K^\infty;\R) \times H^1(X_K^\infty,\partial X_K^\infty;\R) \to H^2(X_K^\infty,\partial X_K^\infty;\R) \cong~\R$$
 defines a nonsingular skew-symmetric $\R$-bilinear form~\cite[Assertion 9]{MilnorInfiniteCyclic}. 
%In fact Milnor proves this result for more general manifolds (Assertion $9$). 
Since the canonical inclusion $(X_K,\emptyset) \to (X_K,\partial X_K)$ induces an isomorphism $H^1(X_K^\infty;\R) \to H^1(X_K^\infty,\partial X_K^\infty;\R)$, the aforementioned cup product pairing gives to rise to a nonsingular skew-symmetric form 
$$\cup \colon H^1(X_K^\infty;\R) \times H^1(X_K^\infty;\R) \to \R.$$
Use $t^*$ to denote the automorphism induced on $H^1(X_K^\infty;\R)$ by the generator of the deck transformation group of $X_K^\infty$.
Milnor defines the \emph{quadratic form of $K$} as the pairing 
\begin{align*}
b_K \colon H^1(X_K^\infty;\R) \times &H^1(X_K^\infty;\R) \to \R \\
&(x,y) \mapsto (t^*x)\cup y+(t^*y) \cup x.
 \end{align*}
This pairing is symmetric and nonsingular~\cite[Assertion 10]{MilnorInfiniteCyclic} and Milnor defines the signature of~$K$ as the signature of $b_K$. Erle later related $\operatorname{sign}(b_K)$ to the Murasugi signature of $K$~\cite{Erle}:
 
\begin{theorem}
\label{thm:Erle}
Let $K$ be an oriented knot. The signature of the symmetric form $b_K$ is equal to the Murasugi signature of $K$:
$$ \operatorname{sign}(b_K)=\sigma(K). $$
\end{theorem}

Next, we describe the so-called Milnor signatures.
Since $\R$ is a field, the ring $\R[t^{\pm 1}]$ is a PID and therefore the torsion $\R[t^{\pm 1}]$-module $H:=H_1(X_K^\infty;\R)$ decomposes as a direct sum over its $p(t)$-primary components, where $p(t)$ ranges over the irreducible polynomials of $\R[t^{\pm 1}]$. \footnote{Here by a \emph{$p(t)$-primary component}, we mean
$ H_{p(t)}=\lbrace x \in H \ | \ p(t)^n x=0 \text{ for some } n> 0 \rbrace$. Observe that~$H_{p(t)} \neq 0$ only if $p(t)$ is a factor of $\Delta_K(t)$.}
As explained in~\cite[proof of Assertion 11]{MilnorInfiniteCyclic}, the symmetric form~$b_K$ decomposes orthogonally once we distinguish symmetric polynomials (i.e. $p(t)=rt^{\pm i} p(t^{-1})$; written $p(t)\stackrel{.}{=}p(t^{-1})$) from non-symmetric ones:
$$(H^1(X_K^\infty;\R),b_K)=\bigoplus_{p(t)\stackrel{.}{=}p(t^{-1})} (H_{p(t)},b_K|_{H_{p(t)}}) \oplus \bigoplus_{p(t) \stackrel{.}{\neq} p(t^{-1})} (H_{p(t)} \oplus H_{p(t^{-1})} , b_K|_{H_{p(t)} \oplus H_{p(t^{-1})}} ).$$
In a nutshell, for $p(t)$ irreducible and symmetric, the restrictions of $b_K|_{H_{p_\theta}(t)}$ produce additional signature invariants. 
%In fact, Milnor argues that these signatures are only non-zero if $p(t)$ is symmetric. 
%Every class $[p(t)]$ is represented by a symmetric polynomial, i.e. a polynomial satisfying $p(t)=p(t^{-1})$.
If $p(t)$ and $q(t)$ differ by multiplication by a unit, then their corresponding primary summands are equal.
From now on, a polynomial is therefore understood to be symmetric if $p(t)=p(t^{-1})$.
As we are working over~$\R[t^{\pm 1}]$, the irreducible symmetric polynomials are of the form $p_\theta(t)=t-2\operatorname{cos}(\theta)+t^{-1}$ with $0<\theta <\pi$.
%No $0,\pi$ because $\Delta_K(\pm 1) \neq 0$...also not irreducible for t=0,\pi it seems.
%\textcolor{blue}{Continue improving below (and check that above has not changed too much.)}

\begin{definition}
\label{def:MilnorSignature}
For $0<\theta<\pi$, the \emph{Milnor signature} $\sigma_\theta(K)$ is the signature of the restriction of $b_K$ to the $p_\theta(t)$-primary summand of $H:=H^1(X_K^\infty;\R)$:
$$\sigma_\theta(K):=\operatorname{sign}(b_K|_{H_{p_\theta(t)}}).$$
\end{definition}

Note that $\sigma_\theta(K)$ is zero if $p_\theta(t)$ does not divide the Alexander polynomial $\Delta_K(t)$ of $K$.
In particular, by Erle's result, the Murasugi signature $\sigma(K)$ is equal to the sum of the $\sigma_\theta(K)$ over all $\theta$ such that $p_\theta(t)$ divides $\Delta_K(t)$. Thus, recalling that $\pm 1$ can not be a root of the Alexander polynomial of a knot, one can write
\begin{equation}
\label{eq:MurasugiMilnor}
 \sigma(K)=\sum_{0 < \theta < \pi} \sigma_\theta(K)=\sum_{ \lbrace \theta  \colon  p_\theta | \Delta_K \rbrace  } \sigma_\theta(K). 
 \end{equation}
Next, following Matumoto, we relate the Milnor signatures to the Levine-Tristram signatures~\cite{Matumoto}. First, note that Erle proves a stronger result than the equality $\sigma(K)=\operatorname{sign}(b_K)$: indeed he shows that $b_K$ is represented by $W+W^T$, where $W$ is a nonsingular matrix over $\Z$ which is S-equivalent to a Seifert matrix of $K$; he calls such a matrix a \emph{reduced Seifert matrix}~\cite[Section 3.4]{Erle}. As a consequence, Matumoto considers an arbitrary nonsingular bilinear form on a $\R$-vector space $V$, represented by a matrix $A$ and compares the signature of $(1-\omega)A+(1-\overline{\omega})A^T$ (for $\omega \in S^1$) with the signatures of $A+A^T$ restricted to the $p(t)$-primary summands of $V$ (here $t$ is thought alternatively as an indeterminate and as the $\R$-automorphism $(A^T)^{-1}A$)~\cite{Matumoto}).

A particular case of one of Matumoto's results can be now be stated as follows~\cite[Theorem~2]{Matumoto}.

\begin{theorem}
\label{thm:Matumoto}
Let $K$ be an oriented knot and let $\omega=e^{i \varphi}$ with $0<\varphi \leq \pi$. If the automorphism~$t^*$ is semisimple or if $\omega$ is not a root of $\Delta_K(t)$, then the following equality holds:
$$ \sigma_K(\omega)=\sum_{0 < \theta <\varphi} \sigma_\theta(K)+ \frac{1}{2} \sigma_\varphi(K).$$
%In particular, if $\omega$ is a root and so on (state as Levine).
%$$ \sigma_K(\omega^+)-\sigma_K(\omega^-)=\sigma_\theta(K).$$
\end{theorem}

Observe that if $\omega=e^{i\varphi}$ is not a root of $\Delta_K(t)$, then the Milnor signature $\sigma_\varphi(K)$ vanishes. In particular, since $-1$ is never a root of the Alexander polynomial of a knot, Theorem~\ref{thm:Matumoto} recovers~\eqref{eq:MurasugiMilnor}. The Milnor pairing can also be considered over $\C$ in which case the statement is somewhat different~\cite[Theorem 1]{Matumoto}. 
Informally, Theorem~\ref{thm:Matumoto} states that the Milnor signatures measure the jumps of $\sigma_K \colon S^1 \to \Z$ at the roots of $\Delta_K(t)$ which lie on $S^1$.
The situation for links is more complicated~\cite{KawauchiQuadraticFormLink}; see also~\cite{KawauchiQuadraticForm3Manifolds}.

We conclude by mentioning some further properties of the Milnor signatures.
\begin{remark}
\label{rem:MilnorSignatureProperties}
The Milnor signatures are concordance invariants~\cite[p.129]{MilnorInfiniteCyclic}. Milnor establishes this result by showing that his signatures vanish on slice knots and are additive under connected sums.
A satellite formula for the Milnor signatures is stated without proof in~\cite{KeartonSatel}.
\end{remark}

\subsection{Signatures via the Blanchfield pairing}
\label{sub:Blanchfield}

In this subsection, we review how the Levine-Tristram signature of a knot can be recovered from the Blanchfield pairing. Note that while the Blanchfield pairing is known to determine the S-equivalence type of $K$~\cite{TrotterSequivalence}, the approaches we discuss here are arguably more concrete.
\medbreak

Given an oriented knot $K$, recall that $X_K^\infty$ denotes the infinite cyclic cover of the exterior $X_K$. 
 Since $\Z=\langle t \rangle$ acts on $X_K^\infty$, the homology group $H_1(X_K^\infty;\Z)$ is naturally endowed with a $\Z[t^{\pm 1}]$-module structure. This $\Z[t^{\pm 1}]$-module is called the \emph{Alexander module} and is known to be finitely generated and torsion~\cite{LevineModule}. 
Using $\Q(t)$ to denote the field of fractions of $\Z[t^{\pm 1}]$, the \emph{Blanchfield form} of a knot is a Hermitian and nonsingular sesquilinear pairing 
$$\operatorname{Bl}_K \colon H_1(X_K^\infty;\Z) \times H_1(X_K^\infty;\Z) \to \Q(t)/\Z[t^{\pm 1}].$$
In order to define $\operatorname{Bl}_K$, we describe its adjoint $\operatorname{Bl}_K^{\bullet} \colon  H_1(X_K^\infty;\Z)  \to \overline{\operatorname{Hom}_{\Z[t^{\pm 1}]}(H_1(X_K^\infty;\Z),\Q(t)/\Z[t^{\pm 1}])}$ so that $\operatorname{Bl}_K (x,y)=\operatorname{Bl}_K ^{\bullet}(y)(x)$. \footnote{Given a ring $R$ with involution, and given an $R$-module $M$, we denote by $\overline{M}$ the $R$-module that has the same underlying additive group as $M$, but for which the action by $R$ on $M$ is precomposed with the involution on $R$.} Using local coefficients, the Alexander module can be written as~$H_1(X_K;\Z[t^{\pm 1}])$. The short exact sequence $0 \to \Z[t^{\pm 1}] \to \Q(t) \to  \Q(t)/\Z[t^{\pm 1}] \to 0$ of coefficients gives rise to a Bockstein homomorphism $\operatorname{BS} \colon H^1(X_K;\Q(t)/\Z[t^{\pm 1}]) \to H^2(X_K;\Z[t^{\pm 1}] )$. Since the Alexander module is torsion, $\operatorname{BS} $ is in fact an isomorphism.  Composing the map induced by the inclusion $\iota \colon (X_K,\emptyset) \to (X_K,\partial X_K)$ with Poincar\'e duality, $\operatorname{BS}^{-1}$ and the Kronecker evaluation map yields the desired $\Z[t^{\pm 1}]$-linear map:
\begin{align}
\label{eq:BlanchfieldDef}
\operatorname{Bl}_K^\bullet \colon  H_1(X_K;\Z[t^{\pm 1}]) & \stackrel{\iota_*}{\to} H_1(X_K,\partial X_K;\Z[t^{\pm 1}])  \stackrel{\operatorname{PD}}{\to} H^2(X_K;\Z[t^{\pm 1}]) \\
 & \stackrel{\operatorname{BS}^{-1}}{\longrightarrow} H^1(X_K;\Q(t)/\Z[t^{\pm 1}]) \stackrel{\operatorname{ev}}{\to} \overline{\operatorname{Hom}_{\Z[t^{\pm 1}]}(H_1(X_K;\Z[t^{\pm 1}]),\Q(t)/\Z[t^{\pm 1}])}. \nonumber
 \end{align}
Following Kearton~\cite{KeartonSignature, KeartonQuadratic}, we outline how 
signatures can be extracted from the (real) Blanchfield pairing
$\Bl_K \colon H_1(X_K;\R[t^{\pm 1}]) \times H_1(X_K;\R[t^{\pm 1}]) \to \R(t)/\LR$.
Let~$p(t)$ be a real irreducible symmetric factor of $\Delta_K(t)$, and let $H_{p(t)}$ be the $p(t)$-primary summand of $H:=H_1(X_K;\R[t^{\pm 1}])$. 
There is a decomposition $H_{p(t)}=\bigoplus_{i=1}^m H_{p(t)}^r$, where each $H_{p(t)}^r$ is a free module over $\LR/p(t)^r \LR$. For $i=1,\ldots, m$, consider the quotient
$$ V_{p(t)}^r:=H_{p(t)}^r/p(t)H_{p(t)}^r$$
as a vector space over $\C \cong \R(\xi) \cong \LR/p(t)\LR$, where $\xi$ is a root of $p(t)$.
The Blanchfield pairing $\Bl_K$ now induces the following well defined Hermitian pairing:
\begin{align*}
\bl_{r,p(t)}(K) \colon V_{p(t)}^r \times V_{p(t)}^r &\to \C \\
([x],[y]) &\mapsto \Bl_K(p(t)^{r-1}x,y).
\end{align*}
As above, we write $p_\theta(t)=t-2\operatorname{cos}(\theta)+t^{-1}$: this way for each $\theta \in (0,\pi)$ and every integer $r$, we obtain additional signature invariants.

\begin{definition}
\label{def:BlanchfieldSignature}
For $0<\theta <\pi$ and $r>0$, the \emph{Blanchfield signature} $\sigma_{r,\theta}(K)$ is the signature of the Hermitian pairing $\bl_{r,p_{\theta}(t)}(K)$:
$$\sigma_{r,\theta}(K):=\operatorname{sign}(\bl_{r,p_{\theta}(t)}(K)).$$
\end{definition}

Kearton~\cite[Section~9]{KeartonSignature} relates the signatures $\sigma_{r,\theta}(K)$ to the Milnor signatures, while Levine relates the $\sigma_{r,\theta}(K)$ to the Levine-Tristram signature~\cite[Theorem 2.3]{LevineMetabolicHyperbolic}:
 
\begin{theorem}
\label{thm:BlanchfieldSignaturesKearton}
Given an oriented knot $K$ and $0<\theta<\pi$, one has
$$ \sigma_\theta(K)=\sum_{r \text{ odd}} \sigma_{r,\theta}(K).$$
If $\omega:=e^{i\theta}$ is a root of $\Delta_K(t)$, and if $\omega_+,\omega_- \in S^1_* \setminus \lbrace \omega \in S^1_* \ | \ \Delta_K(\omega)=0 \rbrace$ are such that $\omega$ is the only root of~$\Delta_K(t)$ lying on an arc of~$S^1$ connecting them, then
%Then if we choose $p(t)$ such that $p(\xi_+)>0$, we have:
%There might be a minus sign: $1/2*p(\xi_+)=cos(\theta_+)-cos(\theta)<0$; the does not quite match with Kearton/Milnor. 
%Since the intersection form and Blanchfield have opposite signs; I suspect there is a minus.
\begin{align*}
&\sigma_K(\omega^+)-\sigma_K(\omega^-)=2 \sum_{r \text{ odd}} \sigma_{r,\theta}(K), \\
&\sigma_K(\omega)=\frac{1}{2}(\sigma_K(\omega^+)-\sigma_K(\omega^-))-\sum_{r \text{ even}} \sigma_{r,\theta}(K).
\end{align*}
\end{theorem}

As we already mentioned in the previous subsection, Theorem~\ref{thm:BlanchfieldSignaturesKearton} (and Theorem~\ref{thm:Matumoto}) shows that the Blanchfield and Milnor signatures measure the jumps of $\sigma_K \colon S^1 \to \Z$ at the roots of~$\Delta_K(t)$.

Next, we mention some further properties of the Blanchfield  signatures. 
\begin{remark}
\label{rem:BlanchfieldSignatureProperties}
For each $0<\theta<\pi$, the sum $\sum_{r \text{ odd}} \sigma_{r,\theta}(K)$ of Blanchfield signatures is a concordance invariant: this can either be seen directly~\cite{LevineMetabolicHyperbolic} or by relating this sum to the Milnor signature~$\sigma_\theta(K)$ (recall Theorem~\ref{thm:BlanchfieldSignaturesKearton}) and using its concordance invariance (recall Remark~\ref{rem:MilnorSignatureProperties}). 
Combining this fact with Theorem~\ref{thm:BlanchfieldSignaturesKearton} yields a proof that the Levine-Tristram signature function~$\sigma_K$ vanishes away from the roots of $\Delta_K$ if $K$ is (algebraically) slice (recall Subsection~\ref{sub:Concordance}).

While the Blanchfield signatures $\sigma_{r,\theta}(K)$ are not concordance invariants, they do vanish if~$K$ is doubly slice~\cite{KeartonDouble, LevineMetabolicHyperbolic}. Combining this fact with Theorem~\ref{thm:BlanchfieldSignaturesKearton} yields a proof that the Levine-Tristram signature function $\sigma_K$ vanishes identically if~$K$ is doubly slice (recall Remark~\ref{rem:DoublySlice}).
\end{remark}

Next,  following Borodzik-Friedl, we describe a second way of extracting signatures from the Blanchfield pairing~\cite{BorodzikFriedl}.
The Blanchfield pairing is known to be \emph{representable}: as shown in~\cite[Proposition 2.1]{BorodzikFriedl} there exists a non-degenerate Hermitian matrix $A(t)$ over $\Z[t^{\pm 1}]$ such that $\Bl_K$ is isometric to the pairing
\begin{align*}
\lambda_{A(t)} \colon  \operatorname{coker}(A(t)^T) \times \operatorname{coker}(A(t)^T)   &\to \Q(t)/\Z[t^{\pm 1}] \\
([x],[y]) & \mapsto x^TA(t)^{-1}\overline{y}.
\end{align*}
In this case, we say that the Hermitian matrix $A(t)$ \emph{represents} $\Bl_K$. 
%The Blanchfield pairing of a knot is a linking form over $R=\Z[t^{\pm 1}]$, is known to be representable~\cite[Proposition 2.1]{BorodzikFriedl} and 
These representing matrices provide an alternative way of defining the Levine-Tristram signature~\cite[Lemma 3.2]{BorodzikFriedl}:

\begin{proposition}
\label{prop:BorodzikFriedl}
%The Blanchfield pairing $\operatorname{Bl}_K$ of a knot $K$ is representable and 
Let $K$ be an oriented knot and let $\omega \in S^1$. For any Hermitian matrix~$A(t)$ which represents the Blanchfield pairing $\operatorname{Bl}_K$, the following equalities hold:
\begin{align*}
&\sigma_K(\omega)=\operatorname{sign}(A(\omega))-\operatorname{sign}(A(1)), \\
&\eta_K(\omega)=\operatorname{null}(A(\omega)).
\end{align*}
\end{proposition}

In the case of links, even though the Blanchfield form can be defined in a way similar to~\eqref{eq:BlanchfieldDef}, no generalization of Proposition~\ref{prop:BorodzikFriedl} appears to be known at the time of writing. Similarly, the Blanchfield signatures described in Definition~\ref{def:BlanchfieldSignature} do not appear to have been generalized to links.

\section{Two additional constructions}
\label{sec:Other}

We briefly discuss two additional constructions of the Levine-Tristram signature. In Subsection~\ref{sub:CassonLin}, we review a construction (due to Lin~\cite{Lin}) which expresses the Murasugi signature of a knot as a signed count of traceless $\SU(2)$-representations. In Subsection~\ref{sub:GG}, we discuss Gambaudo and Ghys' work, a corollary of which expresses the Levine-Tristram signature in terms of the Burau representation of the braid group and the Meyer cocycle.

\subsection{The Casson-Lin invariant}
\label{sub:CassonLin}

Let $K$ be an oriented knot. Inspired by the construction of the Casson invariant, Lin defined a knot invariant $h(K)$ via a signed count of conjugacy classes of traceless irreducible representations of $\pi_1(X_K)$ into $\operatorname{SU}(2)$~\cite{Lin}. Using the behavior of $h(K)$ under crossing changes, Lin additionally showed that $h(K)$ is equal to half the Murasugi signature~$\sigma(K)$. The goal of this subsection is to briefly review Lin's construction and to mention some later generalizations.
\medbreak
Let $X$ be a topological space. The \emph{representation space} of $X$ is the set $R(X):=\operatorname{Hom}(\pi_1(X),\SU(2))$ endowed with the compact open topology. A representation is \emph{abelian} if its image is an abeliean subgroup of $\SU(2)$ and we let~$S(X)$ denote the set of abelian representations. Note that an $\SU(2)$-representation is abelian if and only if it is reducible. The group $\SU(2)$ acts on $R(X)$ by conjugation and its turns out that $\operatorname{SO}(3)=\SU(2)/ \pm \id$ acts freely and properly on the set~$R(X) \setminus S(X)$ of irreducible (i.e. non abelian) representations. The space of conjugacy classes of irreducible~$\SU(2)$-representations of $X$ is denoted by 
$$ \widehat{R}(X)=(R(X) \setminus S(X))/\operatorname{SO}(3).$$
Given an oriented knot $K$ whose exterior is denoted $X_K$, the goal is now to make sense of a signed count of the elements $\widehat{R}(X_K)$. The next paragraphs outline the idea underlying Lin's constrution.

The braid group $B_n$ can be identified with the group of isotopy classes of orientation preserving homeomorphisms of the punctured disk $D_n$ that fix the boundary pointwise. In particular, each braid $\beta$ can be represented by a homeomorphism $h_\beta \colon D_n \to D_n$ which in turn induces an automorphism of the free group $F_n \cong\pi_1(D_n)$. In turn, since $R(D_n) \cong \SU(2)^n$, the braid $\beta$ gives rise to a self-homeomorphism $\beta \colon \SU(2)^n \to \SU(2)^n$. We can therefore consider the spaces
\begin{align*}
&\Lambda_n =\lbrace (A_1,\ldots,A_n,A_1,\ldots,A_n) \ | \ A_i \in \operatorname{SU}_2^n, \ \operatorname{tr}(A_i)=0 \rbrace, \\
&\Gamma_n =\lbrace (A_1,\ldots,A_n,\beta(A_1),\ldots,\beta(A_n)) \ | \ A_i \in \operatorname{SU}_2^n, \ \operatorname{tr}(A_i)=0 \rbrace.
\end{align*}
Use $\widehat{\beta}$ to denote the link obtained as the closure of a braid $\beta$. The representation space  $R^0(X_{\widehat{\beta}})$ of traceless $\SU(2)$ representations of $\pi_1(X_{\widehat{\beta}})$ can be identified with $\Lambda_n \cap \Gamma_n$ i.e. the fixed point set of the homeomorphism $\beta \colon \SU(2)^n \to \SU(2)^n$~\cite[Lemma 1.2]{Lin}. Therefore, Lin's idea is to make sense of an algebaic intersection of $\Lambda_n$ with $\Gamma_n$ inside the ambient space
$$ H_n=\lbrace (A_1,\ldots,A_n,B_1,\ldots,B_n) \in \SU(2)^n \times \SU(2)^n, \ \tr(A_i)=\tr(B_i)=0 \rbrace. $$
Next, we briefly explain how Lin manages to make sense of this algebraic intersection number. The space $\SU(2) \cong S^3$ is $3$-dimensional and the subspace of traceless matrices is homeomorphic to a $2$-dimensional sphere. As a consequence, $\Lambda_n$ and $\Gamma_n$ are both $2n$-dimensional smooth compact manifolds, and Lin shows that $\widehat{H}_n$ is $4n-3$ dimensional~\cite[Lemma 1.5]{Lin}. The $\operatorname{SO}(3)$ action descends to the spaces~$\Lambda_n,\Gamma_n,H_n$ and one sets
$$ \widehat{H}_n=H_n/\operatorname{SO}(3), \ \ \widehat{\Lambda}_n=\Lambda_n/\operatorname{SO}(3), \ \ \widehat{\Gamma}_n=\Gamma_n/\operatorname{SO}(3).\ \ $$
After carefully assigning orientations to these spaces, it follows that $\widehat{\Lambda}_n, \widehat{\Gamma}_n$ are half dimensional smooth oriented submanifolds of the smooth oriented manifold $\widehat{H}_n$. The intersection $\widehat{\Lambda}_n \cap  \widehat{\Gamma}_n$ is compact whenever $\widehat{\beta}$ is a knot~\cite[Lemma 1.6]{Lin} and therefore, after arranging transversality, one can define the \emph{Casson-Lin invariant} of the braid $\beta$ as the algebraic intersection
$$h(\beta):=\langle \widehat{\Lambda}_n,  \widehat{\Gamma}_n \rangle_{\widehat{H}_n}.$$
Lin proves the invariance of $h(\beta)$ under the Markov moves and shows that the resulting knot invariant is equal to half the Murasugi signature~\cite[Theorem 1.8 and Corollary 2.10]{Lin}:

\begin{theorem}
\label{thm:LinTheorem}
The Casson-Lin invariant $h(\beta)$ is unchanged under the Markov moves and thus, setting $h(K)=h(\widehat{\beta})$ for any braid $\beta$ such that $K=\widehat{\beta}$ defines a knot invariant. Furthermore,~$h(K)$ is equal to half the Murasugi signature of $K$:
%\begin{equation}
%\label{eq:Lin}
$$h(K)=\frac{1}{2}\sigma(K).$$
%\end{equation}
\end{theorem}

Lin's work was later generalized by Herald~\cite{Herald} and Heusener-Kroll~\cite{HeusenerKroll} to show that the Levine-Tristram signature~ $\sigma_K(e^{2i\theta})$ can be obtained as a signed count of conjugacy classes of irreducible $\operatorname{SU}(2)$-representations with trace $2\operatorname{cos}(\theta)$. 
Herald obtained this result via a gauge theoretic interpretation of the Casson-Lin invariant (to do so, he used a $4$-dimensional interpretation of the signature), while Heusener-Kroll generalized Lin's original proof (which studies the behavior of~$h(K)$ under crossing changes and uses Remark~\ref{rem:KnotsUnknotting}). 

We also refer to~\cite{HeusenerOrientation} for an interpretation of Lin's construction using the plat closure of a braid (the result is closer to Casson's original construction in terms of Heegaard splittings~\cite{AkbulutMcCarthy}), and to~\cite{CollinSteer} for a construction of an instanton Floer homology theory whose Euler characteristic is the Levine-Tristram signature. 
Is there a formula for links? Can Theorem~\ref{thm:LinTheorem} be understood using the constructions of Section~\ref{sec:3DDefOther}?
%Finally, note that for parity reasons (recall the second item of Proposition~\ref{prop:PropertiesSignature}), Theorem~\ref{thm:LinTheorem} can not hold for arbitrary links.

\subsection{The Gambaudo-Ghys formula}
\label{sub:GG}

Since the Alexander polynomial can be expressed using the Burau representation of the braid group~\cite{Burau}, one might wonder whether a similar result holds for the Levine-Tristram signature. This subsection describes work of Gambaudo and Ghys~\cite{GambaudoGhys}, a consequence of which answers this question in the positive.
\medbreak
Let $B_n$ denote the $n$-stranded braid group. Given $\omega \in S^1$, Gambaudo and Ghys study the map $B_n \to \Z, \beta \mapsto \sigma_{\widehat{\beta}}(\omega)$ obtained by sending a braid to the Levine-Tristram signature of its closure. While this map is not a homomorphism, these authors express the homomorphism defect $\sigma_{\widehat{\alpha\beta}}(\omega)-\sigma_{\widehat{\alpha}}(\omega)-\sigma_{\widehat{\beta}}(\omega)$ in terms of the reduced Burau representation
$$\overline{ \mathcal{B}}_t \colon B_n\to\mathit{GL}_{n-1}(\Z[t^{\pm 1}])\,.$$
We briefly recall the definition of $\overline{ \mathcal{B}}_t$. Any braid $\beta \in B_n$ can be represented by (an isotopy class of) a homeomorphism $h_\beta \colon D_n \to D_n$ of the punctured disk $D_n$. This punctured disk has a canonical infinite cyclic cover $D_n^\infty$ (corresponding to the kernel of the map $\pi_1(D_n) \to \Z$ sending the obvious generators of $\pi_1(D_n)$ to $1$) and, after fixing basepoints, the homeomorphism~$h_\beta$ lifts to a homeomorphism $\widetilde{h}_\beta \colon D_n^\infty \to D_n^\infty$. It turns out that $H_1(D_n^\infty;\Z)$ is a free $\Z[t^{\pm 1}]$-module of rank $n-1$ and the \emph{reduced Burau representation} is the $\Z[t^{\pm 1}]$-linear automorphism of $H_1(D_n^\infty;\Z)$ induced by~$\widetilde{h}_\beta$. This representation is unitary with respect to the equivariant skew-Hermitian form on $H_1(D_n^\infty;\Z)$ which is defined by mapping $x,y \in H_1(D_n^\infty;\Z)$ to 
$$ \xi (x,y)=\sum_{n \in \Z} \langle x,t^n y \rangle t^{-n}. $$
In particular, evaluating any matrix for $\overline{ \mathcal{B}}_t(\beta)$ at $t=\omega$, the matrix $\overline{\mathcal{B}}_\omega(\alpha)$ preserves the skew-Hermitian form obtained by evaluating a matrix for $\xi$ at $t=\omega$. Therefore, given two braids~$\alpha,\beta\in~B_n$ and~$\omega \in S^1$, one can consider the Meyer cocycle of the two unitary matrices~$\overline{\mathcal{B}}_{\omega}(\alpha)$ and~$\overline{\mathcal{B}}_{\omega}(\beta)$. Here, given a skew-Hermitian form $\xi$ on a complex vector space $\C$ and two unitary automorphisms $\gamma_1,\gamma_2$ of $(V,\xi)$, the \emph{Meyer cocycle} $\operatorname{Meyer}(\gamma_1,\gamma_2)$ is computed by considering the space $ E_{\gamma_1,\gamma_2}=\operatorname{im}(\gamma_1^{-1}-\id)\cap \operatorname{im}(\id-\gamma_2)$ and taking the signature of the Hermitian form obtained by setting~$b(e,e')=\xi(x_1+x_2,e')$ for~$e=\gamma_1^{-1}(x_1)-x_1=x_2-\gamma_2(x_2)\in E_{\gamma_1,\gamma_2}$~\cite{Meyer, MeyerTwisted}.

The following result is due to Gambaudo and Ghys~\cite[Theorem A]{GambaudoGhys}.

\begin{theorem}
\label{thm:GG}
For all~$\alpha,\beta\in B_n$ and~$\omega\in S^1$ of order coprime to~$n$, the following equation holds:
\begin{equation}
\label{equ:GG}
\sigma_{\widehat{\alpha\beta}}(\omega)-\sigma_{\widehat{\alpha}}(\omega)-\sigma_{\widehat{\beta}}(\omega)=-\operatorname{Meyer}(\overline{\mathcal{B}}_\omega(\alpha),\overline{\mathcal{B}}_\omega(\beta)).
\end{equation}
\end{theorem}

In fact, since both sides of~\eqref{equ:GG} define locally constant functions on $S^1$, Theorem~\ref{thm:GG} holds on a dense subset of~$S^1$. 
The proof of Theorem~\ref{thm:GG} is 4-dimensional; can it also be understood using the constructions of Section~\ref{sec:3DDefOther}? The answer ought to follow from~\cite{Bourrigan}, where a result analogous to Theorem~\ref{thm:GG} is established for Blanchfield pairings; see also~\cite{GhysRanickiEnsaios}.

We conclude  this survey by applying Theorem~\ref{thm:GG} recursively in order to provide a formula for the Levine-Tristram signature purely in terms of braids. 
Indeed, using $\sigma_1,\ldots,\sigma_{n-1}$ to denote the generators of the braid group~$B_n$ (and recalling that the signature vanishes on trivial links), the next result follows from Theorem~\ref{thm:GG}:

\begin{corollary}
\label{cor:GG}
If an oriented link $L$ is the closure of a braid $\sigma_{i_1}\cdots \sigma_{i_l}$, then the following equality holds on a dense subset of $S^1$:
$$ \sigma_L(\omega)=-\sum_{j=1}^{l-1} \operatorname{Meyer}(\overline{\mathcal{B}}_\omega(\sigma_{i_1}\cdots \sigma_{i_j}),\overline{\mathcal{B}}_\omega(\sigma_{i_{j+1} })). $$
\end{corollary}

\subsection*{Acknowledgments}
I thank Durham University for its hospitality and was supported by an early Postdoc.Mobility fellowship funded by the Swiss National Science Foundation.

\bibliography{bibliothesis}
\bibliographystyle{plain}

\end{document}